\documentclass[12pt,notitlepage]{amsart}
\usepackage{latexsym,amsfonts,amssymb,amsmath,amsthm} 
\usepackage{color}
\newcommand{\mz}{\ensuremath{\mathbb Z}}
\newcommand{\mr}{\ensuremath{\mathbb R}}

\newcommand{\shortmod}{\ensuremath{\negthickspace \negthickspace \negthickspace \pmod}}

\newcommand{\half}{\ensuremath{ \frac{1}{2}}}

\newcommand{\intR}{\int_{-\infty}^{\infty}}

\newcommand{\thalf}{\tfrac12}

\newcommand{\leg}[2]{\left(\frac{#1}{#2}\right)}

\newcommand{\kappaone}{.4105}
\newcommand{\kappastar}{.4058}

\theoremstyle{plain}		
	\newtheorem{mytheo}{Theorem}[section]
	
	\newtheorem{myprop}[mytheo]{Proposition}
	\newtheorem{mycoro}[mytheo]{Corollary}
     \newtheorem{mylemma}[mytheo]{Lemma}

\theoremstyle{remark}

\numberwithin{equation}{section}
\begin{document}
\title{More than $41\%$ of the zeros of the zeta function are on the critical line}

\author{H. M. Bui}
\address{Mathematical Institute \\ University of Oxford \\ Oxford
OX1 3LB, UK}
\email{hung.bui@maths.ox.ac.uk}
\thanks{H.M.B. supported by an EPSRC Postdoctoral
Fellowship.}

\author{Brian Conrey}
\address{American Institute of Mathematics\\
          360 Portage Ave.\\
          Palo Alto, CA 94306}
\email{conrey@aimath.org}
\thanks{B.C. partially supported by the American Institute of Mathematics.}

\author{Matthew P. Young} 
\address{Department of Mathematics \\
 	  Texas A\&M University \\
 	  College Station \\
	  TX 77843-3368}
\curraddr{School of Mathematics \\
Institute for Advanced Study \\
Einstein Drive \\ Princeton, NJ 08540 USA}
\thanks{This material is based upon work supported by the National Science Foundation under agreement Nos. DMS-0801264 (B.C.) DMS-0758235 (M.Y.), and DMS-0635607 (M.Y.).  Any opinions, findings and conclusions or recommendations expressed in this material are those of the authors and do not necessarily reflect the views of the National Science Foundation.}
\email{myoung@math.tamu.edu}


\begin{abstract}
We prove that at least $41.05 \%$ of the zeros of the Riemann zeta function are on the critical line.
\end{abstract}

\maketitle
\section{Introduction}
The location of the zeros of the Riemann zeta function is one of the most fascinating subjects in number theory.  In this paper we study the percent of zeros lying on the critical line.  With the use of a new two-piece mollifier, we make a modest improvement on this important problem.

To set some terminology, let $N(T)$ denote the number of zeros $\rho = \beta + i\gamma$ with $0 < \gamma < T$, let $N_0(T)$ denote the number of such critical zeros with $\beta = 1/2$, and let $N_0^*(T)$ denote the number of such critical  zeros which are simple.  Define $\kappa$ and $\kappa^*$ by
\begin{equation*}
 \kappa = \liminf_{T \rightarrow \infty} \frac{N_0(T)}{N(T)}, \qquad \kappa^* = \liminf_{T \rightarrow \infty} \frac{N_0^*(T)}{N(T)}.
\end{equation*}

Selberg \cite{Selberg} was the first to prove that a positive percentage of zeros lie on the critical line.  There has since been a series of improvements, of which we briefly mention the work of Levinson \cite{Levinson} obtaining $\kappa \geq .3474$, and the current record of $\kappa \geq .4088$, $\kappa^* \geq .4013$ due to Conrey \cite{Conrey25}.

In this paper we show
\begin{mytheo}
 We have
\begin{equation}
 \kappa \geq \kappaone, \qquad \kappa^* \geq \kappastar.
\end{equation}
\end{mytheo}
Our method is to revisit an old approach of Lou \cite{LouDurham} by taking a two-piece mollifier (meaning the sum of two mollifiers, each of a different shape).  The details of Lou's work never appeared in print and there is some doubt as to its correctness.  We have added innovations to Lou's approach by taking a longer mollifier (requiring delicate analysis of off-diagonal terms) and by combining it with ideas of Conrey \cite{Conrey25} \cite{ConreyJNT}.

S. Feng \cite{Feng} has also introduced a new two-piece mollifier
with the purpose of proving a lower bound for the
proportion $\kappa$ of zeros of $\zeta(s)$ on the critical line.
We have checked that Feng's formulas are correct, except
that we cannot verify that his new mollifier is permitted
to have a length $\theta = 4/7-\epsilon$ as is originally
claimed in his preprint. His new mollifier
taken with length $\theta = 1/2-\epsilon$, which is
admissible, would surely lead to an improvement over Conrey's
$\kappa\ge 0.4088$, but it is not clear to us what
his final bound for $\kappa$ will be.

One of the difficulties in studying this and other problems involving mollifiers is that it takes a significant amount of computation to judge how much progress one makes with a new idea.  However, there are some heuristics that can save a lot of time.  In particular, the ratios conjecture \cite{CFZ} can rather quickly allow one to express mollified moments of $L$-functions as certain multiple contour integrals; see \cite{CS} for a variety of examples of such calculations.  Even then, it takes some significant work to simplify these contour integrals into a form usable for calculation.  With some practice these calculations become routine, and we have made an effort to describe the reasoning behind our approach.

\section{Reduction to mean-value theorems}
\subsection{The setup}
\label{section:setup}
The basic technology to prove that many zeros lie on the critical line is an asymptotic for a mollified 
second moment of the zeta function.  In this section, we recall how to  reduce the problem to such mean value estimates.  This is mostly a summary of \cite{Conrey25}.

Let $\zeta(s) = \sum_{n=1}^{\infty} n^{-s}$ for $s = \sigma + it$, $\sigma > 1$.  The functional equation states
\begin{equation*}
 \xi(s) = \xi(1-s), \quad \text{where} \quad
 \xi(s) =  H(s) \zeta(s),
 \quad \text{and} \quad H(s) = \thalf s(s-1) \pi^{-\frac{s}{2}} \Gamma(\tfrac{s}{2}).
\end{equation*}
In its asymmetrical form the functional equation reads
\begin{equation*}
 \zeta(s) = \chi(s) \zeta(1-s), \quad \text{where} \quad \chi(1-s) = 2 (2\pi)^{-s} \Gamma(s) \cos(\tfrac{\pi s}{2}).
\end{equation*}

To get a lower bound on $N_0(T)$ it suffices to consider a certain mollified second moment of $\zeta$ and its derivatives.  This is well-known; see Section 3 of \cite{Conrey25}, for example, so we shall simply state the conclusion.

Let $Q(x)$ be a 
real polynomial satisfying $Q(0) = 1$, $Q(x) + Q(1-x) = \text{constant}$, and define
\begin{equation*}
 V(s) = Q\Big(-\frac{1}{L} \frac{d}{ds} \Big) \zeta(s),
\end{equation*}
where for large $T$, 
\begin{equation*}
 L = \log{T}.
\end{equation*}
Suppose $\psi(s)$ is a ``mollifier''.  Littlewood's lemma and the arithmetic-mean, geometric-mean inequality give
\begin{equation}
\label{eq:kappa}
 \kappa \geq 1 - \frac{1}{R} \log \left( \frac{1}{T} \int_{1}^{T} |V \psi(\sigma_0 + it)|^2 dt \right) + o(1),
\end{equation}
where $\sigma_0 = \half -\frac{R}{L}$, and $R$ is a bounded positive real number to be chosen later.  Actually, by choosing $Q(x)$ to be a linear polynomial, one obtains a lower bound on the percent of simple zeros, $\kappa^*$.

We choose a mollifier of the form
\begin{equation*}
\psi(s) =  \psi_1(s) + \psi_2(s),
\end{equation*}
where $\psi_1$ and $\psi_2$ are mollifiers of quite different shape.  Here $\psi_1$ is mollifier of a familiar type from \cite{Conrey25}.  Let $P_1(x) = \sum_j a_j x^j$ be a certain polynomial satisfying $P_1(0) = 0$, $P_1(1) = 1$, let $y_1 = T^{\theta_1}$ where $0< \theta_1 < \frac47$, and use the notation
\begin{equation}
\label{eq:P1def}
 P_1[n] = P_1 \Big(\frac{\log y_1/n}{\log{y_1}} \Big),
\end{equation}
for $1 \leq n \leq y_1$.  By convention, we set $P_1[x] = 0$ for $x \geq y_1$.
With this notation, 
\begin{equation*}
 \psi_1(s) = \sum_{n \leq y_1} \frac{\mu(n) P_1[n] 
n^{\sigma_0 - \half}}{n^s}.
\end{equation*}

For the second mollifier, we take
\begin{equation}
\label{eq:psi2def}
 \psi_2(s) = \chi(s+\thalf-\sigma_0) \sum_{hk \leq y_2} \frac{\mu_2(h) h^{\sigma_0 - \half} k^{\half-\sigma_0} }{h^s k^{1-s}} P_2[hk]
\end{equation}
where $\mu_2(h)$ are the coefficients of $1/\zeta^2(s)$ and $P_2(x) = \sum_j b_j x^j$
is a polynomial satisfying $P_2(0) = P_2'(0) = P_2''(0) = 0$.  Here $y_2 = T^{\theta_2}$ where $\theta_2 < \theta_1$ (we shall see later what conditions are required on $\theta_2$). Note that (formally)
\begin{equation*}
\chi(s) \sum_{h, k = 1}^{\infty} \frac{\mu_2(h)}{h^s k^{1-s}} = \frac{\chi(s)\zeta(1-s)}{\zeta^2(s)} = \frac{1}{\zeta(s)},
\end{equation*}
which explains why $\psi_2(s)$ may be a useful choice of a mollifier.  Here $\psi_2(s)$ is of a somewhat similar shape to a mollifier chosen by Shi-tuo Lou in \cite{LouDurham}. 
Lou considered a mollifier of the form $L^{-2} \chi(s)$ times a Dirichlet series that is roughly of the shape above, with the choice of smoothing polynomial $P_2(x) = x$.  Unfortunately, the details of the calculations are omitted, and there is some doubt on whether the result is correct since our analysis suggests that $P_2$ must vanish to third order; the presence of $L^{-2}$ is also suspect.

The method sketched in Section 3 of \cite{Conrey25} carries through with our choice of $\psi= \psi_1 + \psi_2$, but there is one extra ingredient worthy of mention.
To apply Littlewood's lemma, one needs to estimate the integral on the right side of a rectangle, say at $\sigma = 2$.  This can be done using the trivial bound for $\sigma \geq \half$, say
\begin{equation}
 |\psi_2(s)| \ll \sqrt{t} \leg{y_2}{t}^{\sigma} L^2.
\end{equation}
As long as $\theta_2 < 1$ this is small for $\sigma$ sufficiently large.  The point is that $\chi(s)$ is small for large $\sigma$; 
a mollifier of the form $\chi(1-s)$ times
    a Dirichlet polynomial runs into problems
for $\sigma$ large.

\begin{mytheo}
\label{thm:sharp}
Suppose $\theta_1 = 4/7-\varepsilon$ and $\theta_2 = 1/2 - \varepsilon$ for $\varepsilon > 0$ small.  Then
\begin{equation}
 \frac{1}{T} \int_1^{T} |V \psi(\sigma_0 + it)|^2 dt = c(P,Q,R,\theta_1, \theta_2) +o(1),
\end{equation}
where $c(P,Q,R,\theta_1, \theta_2) = c_1 + 2c_{12} + c_2$ and the $c_{i}$ are given below by \eqref{eq:c1}, \eqref{eq:c12}, and \eqref{eq:c2}.
\end{mytheo}

\subsection{Numerical evaluations}
We use Mathematica to numerically evaluate $c(P,Q,R,4/7, 1/2)$ with the following particular choices of parameters.
With $R=1.28$, 
\begin{gather*}
Q(x) = 0.492 + 0.604(1 - 2x) - 0.08(1 - 2x)^3 - 0.06(1 - 2x)^5 + 
 0.046(1 - 2x)^7, \\ 
 P_1(x) = .842706x + .00845721x^2 + .093117x^3 + .118788x^4 -.0630687x^5, \\
 P_2(x) = .0245412x^3 - .00635566x^4+.00603128x^5,
\end{gather*}
we have $\kappa \geq \kappaone$.  To get $\kappa^* \geq \kappastar$, we take $R= 1.12$,  $Q(x) = 1- 1.03x$, 
\begin{gather*}
 P_1(x) = .829473x + .0104358x^2 + .082009x^3 + .177482x^4 -.0993997x^5, \\
P_2(x) = .0323061x^3 -.00553783x^4+.0769594x^5.
\end{gather*}

\subsection{A smoothing argument}
\label{section:smoothing}
It simplifies some calculations to smooth out the integral in \eqref{eq:kappa}.  Suppose $w(t)$ is a smooth function satisfying the following properties:
\begin{align}
\label{eq:w1}
&0 \leq w(t) \leq 1 \text{ for all } t \in \mr, \\
&w \text{ has compact support in } [T/4, 2T], \\
&w^{(j)}(t) \ll_j \Delta^{-j}, \text{ for each } j=0,1,2, \dots, \quad \text{where } \Delta = T/L.
\label{eq:w3}
\end{align}
\begin{mytheo}
\label{thm:smooth}
 For any $w$ satisfying \eqref{eq:w1}-\eqref{eq:w3}, and $\sigma = 1/2 - R/L$, 
\begin{equation}
 \intR w(t) |V \psi(\sigma + it)|^2 dt = c(P,Q,R,\theta_1, \theta_2) \widehat{w}(0) + O(T L^{-1+\varepsilon}),
\end{equation}
uniformly for $R \ll 1$, where $c(P,Q,R, \theta_1, \theta_2) = c_1 + 2c_{12} + c_2$, and where these constants are given below by \eqref{eq:c1}, \eqref{eq:c12}, and \eqref{eq:c2}.
\end{mytheo}
We briefly explain how to deduce Theorem \ref{thm:sharp} from Theorem \ref{thm:smooth}.
By choosing $w$ to satisfy \eqref{eq:w1}-\eqref{eq:w3} and in addition to be an upper bound for the characteristic function of the interval $[T/2, T]$, and with support in $[T/2 - \Delta, T + \Delta]$, we get
\begin{equation}
 \int_{T/2}^{T} |V \psi(\sigma_0 + it)|^2 dt \leq c(P, Q, R, \theta_1, \theta_2) \widehat{w}(0) + O(T L^{-1+\varepsilon}).
\end{equation}
Note $\widehat{w}(0) = T/2 + O(T/L)$.
We similarly get a lower bound.  Summing over dyadic segments gives the full integral.

\section{The mean-value results}
Writing $\psi = \psi_1 + \psi_2$ and opening the square, we get
\begin{equation*}
 \int |V \psi|^2 = \int |V \psi_1|^2 + \int |V|^2 \psi_1 \overline{\psi_2} + \int |V|^2 \overline{ \psi_1} \psi_2 + \int |V \psi_2|^2 =: I_1 + I_{12} + \overline{I_{12}} + I_2.
\end{equation*}
We shall compute the integrals in turn.  It turns out that $I_{12}$ is asymptotically real.

\subsection{The main terms}
Recall the conditions on $Q, P_1$, and $P_2$ stated in Section \ref{section:setup}.

First we quote Theorem 2 of \cite{Conrey25}.
\begin{mytheo}[Conrey]  Suppose $\theta_1 < 4/7$.  Then
\begin{equation}
\int_1^{T}  |V\psi_1(\sigma_0 + it)|^2 dt \sim c_{1}(P_1, Q, R, \theta_1) T,
\end{equation}
as $T \rightarrow \infty$, where
\begin{equation}
\label{eq:c1}
 c_{1}(P_1, Q, R, \theta_1) = 1 + \frac{1}{\theta_1} \int_0^{1} \int_0^{1} e^{2Rv} (Q(v) P_1'(u) + \theta_1 Q'(v) P_1(u) + \theta_1 R Q(v) P_1(u))^2 du dv.
\end{equation}
\end{mytheo}
This is the unsmoothed version, but the smoothed version follows easily from this.

We handle the other terms as follows
\begin{mytheo}
\label{thm:mainterm12}
 Suppose $\theta_2 < \theta_1  < 4/7$.   
Then
\begin{equation}
 I_{12} = \intR w(t) |V|^2 \overline{\psi_1} \psi_2   (\sigma_0 + it) dt = c_{12} \widehat{w}(0) + O(T/L),
\end{equation}
where
\begin{multline}
\label{eq:c12}
c_{12}= 4 \frac{ \theta_2^2}{\theta_1^2} e^{R}
 \frac{d^2}{dx dy}\Big[ \mathop{\int \int}_{\substack{0 \leq a+b \leq 1 \\ 0 \leq a,b}} 
  \int_0^{1} u^2 (1-u) e^{R[\theta_1(y-x) + u\theta_2(a-b)]}  Q(-x \theta_1 + au \theta_2)
\\
\left.   Q(1+y \theta_1 -bu \theta_2) P_1\left(x+y+ 1-(1-u)\frac{ \theta_2}{\theta_1}\right) P_2''((1-a-b)u) du \, da \, db \right]_{x=y=0} .
\end{multline}
\end{mytheo}
\begin{mytheo}
\label{thm:mainterm2}
Suppose $\theta_2 < \half$.  Then
\begin{equation}
 I_{2} = \intR w(t) |V|^2 |\psi_2|^2 (\sigma_0 + it) dt = c_2 \widehat{w}(0) + O(TL^{-1+\varepsilon}),
\end{equation}
where
\begin{multline}
\label{eq:c2}
 c_2 = \frac{2}{3} \frac{d^4}{dx^2 dy^2}
\left[ \int_0^{1} \int_0^{1} \int_0^{1} \int_0^{1} (1-r)^{4} (\frac{1}{\theta_2} + (x+y-v(y+r)-u(x+r))) 
\right.
\\
e^{-\theta_2 R(x+y-v(y+r) -u(x+r))} Q(\theta_2(-y + u(x+r)) + t(1+\theta_2(x+y-v(y+r)-u(x+r)))) 
\\
e^{2Rt(1+\theta_2(x+y-v(y+r)-u(x+r)))}  Q(\theta_2(-x + v(y+r)) + t(1+\theta_2(x+y-v(y+r)-u(x+r))))
\\
\left.  (x+r)(y+r) P_2''\left((1-u)(x + r)\right) P_2''\left((1-v)(y + r)\right) dt \, dr \, du \, dv
\right]_{x=y=0}.
\end{multline}

\end{mytheo}
\noindent {\bf Remark}.  Note that $c_{12}$ is real so that $I_{12} \sim \overline{I_{12}}$.

\subsection{The shift parameters}
\label{section:shifts}
Rather than working directly with $V(s)$, we shall instead consider the following two general integrals:
\begin{equation}
\label{eq:I12def}
 I_{12}(\alpha,\beta) = \intR w(t) \zeta(\thalf + \alpha + it) \zeta(\thalf + \beta -it) \overline{\psi_1} \psi_2(\sigma_0 + it) dt,
\end{equation}
and
\begin{equation}
\label{eq:I2def}
 I_{2}(\alpha,\beta) = \intR w(t) \zeta(\thalf + \alpha + it) \zeta(\thalf + \beta -it) |\psi_2(\sigma_0 + it)|^2 dt.
\end{equation}
Our main goal in the rest of the paper is in proving the following two lemmas.
\begin{mylemma}
\label{lemma:mainterm12}
We have
\begin{equation}
 I_{12}(\alpha, \beta) = c_{12}(\alpha, \beta) \widehat{w}(0) + O(T/L),
\end{equation}
uniformly for $\alpha, \beta \ll L^{-1}$, where
\begin{multline}
 c_{12}(\alpha, \beta) = 4 \frac{ \theta_2^2}{\theta_1^2} \frac{d^2}{dx dy}\left[ \mathop{\int \int}_{0 \leq a+b \leq 1} 
  \int_0^{1} u^2 (1-u) (y_1^{-x} y_2^{au})^{-\alpha} (y_1^{y} y_2^{-ub} T)^{-\beta} 
\right.
\\
\left.  P_1\left(x+y+ 1-(1-u)\frac{ \theta_2}{\theta_1}\right) P_2''((1-a-b)u) du \, da \, db \right]_{x=y=0} .
\end{multline}
\end{mylemma}
\begin{mylemma}
\label{lemma:mainterm2}
We have
\begin{equation}
 I_{2}(\alpha, \beta) = c_{2}(\alpha, \beta) \widehat{w}(0) + O(T L^{-1+\varepsilon}),
\end{equation}
uniformly for $\alpha, \beta \ll L^{-1}$, where
\begin{multline}
\label{eq:c2alpha}
 c_{2}(\alpha, \beta) = \frac{2}{3} \frac{d^4}{dx^2 dy^2}
\left[ \int_0^{1} \int_0^{1} \int_0^{1} \int_0^{1} (1-r)^{4} y_2^{\beta(x-v(y+r)) + \alpha(y-u(x+r))}
\right.
\\
(\frac{1}{\theta_2} + (x+y-v(y+r)-u(x+r))) (Ty_2^{x+y - v(y+r) - u(x+r)})^{-t(\alpha+\beta)} (x+r)(y+r)
\\
\left.   P_2''\left((1-u)(x + r)\right) P_2''\left((1-v)(y + r)\right) dt \, dr \, du \, dv
\right]_{x=y=0}.
\end{multline}
\end{mylemma}
We now prove that Theorems \ref{thm:mainterm12} and \ref{thm:mainterm2} follow from Lemmas \ref{lemma:mainterm12} and \ref{lemma:mainterm2}, respectively.
Let $I_{\star}$ denote either $I_{12}$ or $I_{2}$.  Note
\begin{equation}
\label{eq:diffop}
I_{\star} =   Q\Big(\frac{-1}{\log{T}} \frac{d}{d\alpha}\Big) Q\Big(\frac{-1}{\log{T}} \frac{d}{d\beta}\Big) I_{\star}(\alpha, \beta) \Big|_{\alpha=\beta=-R/L}.
\end{equation}
We first argue that we can obtain either $c_{\star}$ by applying the above differential operator to the corresponding $c_{\star}(\alpha, \beta)$.  Since $I_{\star}(\alpha, \beta)$ and $c_{\star}(\alpha, \beta)$ are holomorphic with respect to $\alpha, \beta$ small, the derivatives appearing in \eqref{eq:diffop} can be obtained as integrals of radii $\asymp L^{-1}$ around the points $-R/L$, using Cauchy's integral formula.  Since the error terms hold uniformly on these contours, the same error terms that hold for $I_{\star}(\alpha, \beta)$ also hold for $I_{\star}$.

Next we check that applying the above differential operator to $c_{\star}(\alpha, \beta)$ does indeed give $c_{\star}$.
Notice the formula
\begin{equation}
\label{eq:Qop}
 Q\Big(\frac{-1}{\log{T}} \frac{d}{d\alpha}\Big) X^{-\alpha} = Q\Big( \frac{\log{X}}{\log{T}} \Big) X^{-\alpha}.
\end{equation}
Using \eqref{eq:Qop}, we have
\begin{multline}
 Q\Big(\frac{-1}{\log{T}} \frac{d}{d\alpha}\Big) Q\Big(\frac{-1}{\log{T}} \frac{d}{d\beta}\Big) c_{12}(\alpha,\beta) = 4 \frac{ \theta_2^2}{\theta_1^2} \frac{d^2}{dx dy}\Big[ \mathop{\int \int}_{0 \leq a +b \leq 1}
\\
\left. 
  \int_0^{1} u^2 (1-u) (y_1^{-x} y_2^{au})^{-\alpha} (y_1^{y} y_2^{-ub}T)^{-\beta} Q(-x \theta_1 + au \theta_2) Q(1+y \theta_1 -bu \theta_2)
\right.
\\
    P_1\Big(x+y+ 1-(1-u)\frac{ \theta_2}{\theta_1}\Big) P_2''((1-a-b)u) du \, da \, db \Big]_{x=y=0} .
\end{multline}
Setting $\alpha = \beta = -R/L$ and simplifying 
gives \eqref{eq:c12}.  A similar argument produces \eqref{eq:c2} from \eqref{eq:c2alpha}.

We prove Lemma \ref{lemma:mainterm12} in Section \ref{section:cross}, and Lemma \ref{lemma:mainterm2} in Section \ref{section:diag}.

\section{Various lemmas}
In this section we gather some miscellaneous results that are more or less standard.  These will be used to calculate the constants $c_{\star}(\alpha, \beta)$, and we place them here to avoid interrupting the forthcoming arguments.

\subsection{Approximate functional equations}
In the calculation of $I_{12}$ we shall need the following approximate functional equation with one long sum.
\begin{mylemma}
\label{lemma:longAFE}
Let $\sigma_{\alpha,-\beta}(l) = \sum_{ab = l} a^{-\alpha} b^{\beta}$.
 For $L^2 \leq |t| \leq 2T$ and uniformly for $\alpha, \beta \ll L^{-1}$,
\begin{equation}
\label{eq:longAFE}
 \zeta(\thalf + \alpha + it) \zeta(\thalf - \beta +it) = \sum_{l=1}^{\infty} \frac{\sigma_{\alpha, -\beta}(l)}{l^{\half + it}} e^{-l/T^3} + O(T^{-1 + \varepsilon}).
\end{equation}
\end{mylemma}
\noindent {\bf Remark.}  The exponential decay effectively means that $l$ is truncated at $T^3$.  An essentially identical proof could truncate the sum at $T^{2 + \varepsilon}$ but it would not help the later arguments.
\begin{proof}[Proof of Lemma \ref{lemma:longAFE}]
Consider the following sum
\begin{equation*}
A= \sum_{l=1}^{\infty} \frac{\sigma_{\alpha, -\beta}(l)}{l^{\half + it}} e^{-l/V},
\end{equation*}
where $V$ is a parameter to be chosen momentarily.  Using the formula $e^{-x} = \frac{1}{2 \pi i} \int_{(1)} \Gamma(s) x^{-s} ds$, we get
\begin{equation*}
 A = \frac{1}{2 \pi i} \int_{(1)} V^{s} \Gamma(s)  \zeta(\thalf + \alpha + it + s) \zeta(\thalf - \beta +it +s) ds.
\end{equation*} 
Next we move the line to $\sigma = -1+\varepsilon$, crossing poles at $s=0, \half -\alpha - it, \half +\beta -it$.  Using the bound $\zeta(\sigma + it) \ll |t|^{\half - \sigma}$ for $\sigma < 0$ and the exponential decay of the gamma function, we get
\begin{equation*}
 A = \zeta(\thalf + \alpha + it) \zeta(\thalf - \beta +it) + O(\sqrt{V} e^{-|t|} + V^{-1 + \varepsilon} |t|^2).
\end{equation*}
Thus taking $V = T^3$ finishes the proof.
\end{proof}

A key step in asymptotically evaluating $I_2$ is finding an asymptotic for the ``twisted'' second moment of zeta.  This is essentially implicit in \cite{BCHB} but we cannot quite quote the result we need.  Nevertheless, we state the following without proof
\begin{myprop}
\label{prop:twisted}
Suppose $w$ satisfies \eqref{eq:w1}-\eqref{eq:w3}, and $a$ and $b$ are positive integers with $ab \leq T^{1-\varepsilon}$. Then uniformly for $\alpha, \beta \ll L^{-1}$, we have
\begin{multline}
 \intR \leg{a}{b}^{-it} w(t) \zeta(\thalf + \alpha + it) \zeta(\thalf + \beta - it) dt 
= 
\sum_{am = b n} \frac{1}{m^{\half + \alpha} n^{\half + \beta}} \intR V_t(mn) w(t) dt
\\
+ \sum_{a m = b n} \frac{1}{m^{\half - \beta} n^{\half - \alpha}} \intR V_t(mn) \leg{t}{2 \pi}^{-\alpha - \beta} w(t) dt + O(T^{-1/2}).
\end{multline}
Here $V_t(x)$ is given by
\begin{equation*}
 V_t(x) = \frac{1}{2\pi i} \int_{(1)} \leg{t}{2 \pi x}^{z} \frac{G(z)}{z} dz,
\end{equation*}
where $G(z) = e^{z^2} p(z)$ and $p(z) = \frac{(\alpha+\beta)^2-(2z)^2}{(\alpha + \beta)^2}$.
\end{myprop}
\noindent {\bf Remarks.}  $G(z)$ can be chosen from a wide class of functions; we made this choice since it has rapid decay and vanishes at $2z = \pm (\alpha + \beta)$.  This result is ``easy'' in the sense that with $ab \leq T^{1-\varepsilon}$ only the diagonal terms contribute to the main term (one easily bounds the off-diagonal terms by repeated integration by parts).  Also note $\leg{t}{2\pi}^{-\alpha-\beta} = T^{-\alpha-\beta}(1+ O(L^{-1}))$ for $t \asymp T$, which is implied by the support of $w$.

\subsection{Exercises with Euler-Maclaurin}
We need to evaluate various sums to which we apply the Euler-Maclaurin formula.  We collect these formulas here.

\begin{mylemma}
 \label{lemma:EulerMac}
Suppose $l$ is a nonnegative integer, $x \geq 1$ is real, and $s$ is a complex number with $|s| \leq (\log{x})^{-1}$. Then
\begin{equation}
\sum_{n \leq x} \frac{1}{n^{1 + s}}(\log (x/n))^{l} = (\log{x})^{l+1} x^{-s} \int_0^{1} x^{s a} a^l da + O((\log{3x})^l).
\end{equation}
\end{mylemma}
\begin{proof}[Proof of Lemma \ref{lemma:EulerMac}]
A simple application of the Euler-Maclaurin formula gives
\begin{equation*}
 \sum_{n \leq x} \frac{1}{n^{1 + s}}(\log (x/n))^{l} = \int_1^{x} t^{-1-s} (\log (x/t))^{l} dt + O((\log{3x})^l).
\end{equation*}
We make the change of variables $t = x^{1-a}$ to obtain the desired expression.
\end{proof}

\begin{mylemma}
 \label{lemma:EulerMacCross}
Suppose $z \leq x$, $|s| \leq (\log{x})^{-1}$, $k$ is a positive integer, and that $F$ and $H$ are (fixed) smooth functions.  Then
\begin{multline}
\label{eq:dkFH}
 \sum_{n \leq z} \frac{d_k(n)}{n^{1+s}} F\leg{\log(x/n)}{\log{x}} H\leg{\log(z/n)}{\log{z}} 
\\
= \frac{(\log z)^k}{(k-1)!} z^{-s} \int_0^{1} (1-u)^{k-1} F\left(1-\frac{(1-u) \log z}{\log{x}}\right) H(u) z^{us} du + O((\log{3z})^{k-1}).
\end{multline}
\end{mylemma}
\noindent {\bf Remark.}  By repeatedly using Euler-Maclaurin we can express this as a $k$-fold integral.  It turns out that this multiple integral can be computed explicitly enough to reduce to a single integral as above.

\begin{proof}
We proceed by induction.  For $k=1$, a minor variation of the proof of Lemma \ref{lemma:EulerMac} gives that \eqref{eq:dkFH} is
\begin{equation*}
(\log{z}) z^{-s} \int_0^{1}  F\Big(\frac{\log{x} - (1-u) \log z}{\log{x}}\Big) H(u) z^{us} du + O(1),
\end{equation*}
as desired.  Now suppose $k \geq 2$.  Write the left hand side of \eqref{eq:dkFH} as
\begin{equation}
\label{eq:dkFHstep1}
 \sum_{n \leq z} \frac{1}{n^{1+s}} \sum_{m \leq n^{-1} z} \frac{d_{k-1}(m)}{m^{1+s}} F\Big(\frac{\log(x/n)}{\log{x}} \frac{\log(n^{-1}x/m)}{\log(n^{-1}x)} \Big) H\Big(\frac{\log(z/n)}{\log{z}} \frac{\log(n^{-1}z/m)}{\log{n^{-1}z}} \Big).
\end{equation}
By applying the induction hypothesis to the inner sum over $m$ we get that \eqref{eq:dkFHstep1} is
\begin{multline}
\label{eq:dkFHstep2}
 \int_0^{1} z^{us-s} \frac{(1-u)^{k-2}}{(k-2)!} \sum_{n \leq z} \frac{(\log \frac{z}{n})^{k-1}}{n^{1+us}}  F\Big(\frac{\log(x/n)}{\log{x}} \Big( 1-\frac{(1-v) \log (z/n)}{\log(x/n)}\Big)\Big) H\Big(\frac{\log(z/n)}{\log{z}}v\Big)  dv
\\
+O((\log{3z})^{k-1}).
\end{multline}
We again apply a minor variation on Lemma \ref{lemma:EulerMac}, this time to the sum over $n$ which gives that \eqref{eq:dkFHstep2} is
\begin{multline}
z^{-s} (\log{z})^k \int_0^{1} \frac{(1-v)^{k-2}}{(k-2)!} \int_0^{1}z^{-usv} u^{k-1} F\Big(1-(1-u)\frac{\log{z}}{\log{x}} - \frac{(1-v)u\log{z}}{\log{x}} \Big) H(u v )  du dv
\\
+O((\log{3z})^{k-1}).
\end{multline}
Now we perform some elaborate changes of variables: first $v \rightarrow 1-v$, followed by $v \rightarrow v/u$, and finally $u \rightarrow u+v$ to give
\begin{equation*}
z^{-s} \frac{(\log{z})^k}{(k-2)!} \mathop{\int \int}_{\substack{u + v \leq 1 \\ 0 \leq u, v}} v^{k-2}  z^{us}  F\Big(1-\frac{\log{z}}{\log{x}} + \frac{u\log{z}}{\log{x}} \Big) H(u)  du dv
+O((\log{3z})^{k-1}).
\end{equation*}
The integral over $v$ can now be calculated to finish the proof.
\end{proof}
In the special case $z=x$ we obtain
\begin{mycoro}
\label{coro:EulerMacdiag}
 Let assumptions be as in Lemma \ref{lemma:EulerMacCross}.  Then
\begin{multline}
\label{eq:dkFHz=x}
 \sum_{n \leq x} \frac{d_k(n)}{n^{1+s}} F\Big( \frac{\log(x/n)}{\log{x}} \Big) H\Big(\frac{\log(x/n)}{\log{x}} \Big)
\\
= \frac{(\log x)^k}{(k-1)!} x^{-s} \int_0^{1} (1-u)^{k-1} F(u) H(u) x^{us} du + O((\log{3x})^{k-1}).
\end{multline}
\end{mycoro}
We finally need one other result in this circle of sums.  This result is important for us because it saves one $\log$
when $|\sigma|$ is taken to be a fixed positive constant.
\begin{mylemma}
\label{lemma:logsave}
Suppose $-1 \leq \sigma \leq 0$.  Then
\begin{equation}
\label{eq:dkUB}
\sum_{n \leq x} \frac{d_k(n)}{n} \leg{x}{n}^{\sigma}  \ll_k (\log{3x})^{k-1} \min(|\sigma|^{-1}, \log{3x}).
\end{equation}
\end{mylemma}
\begin{proof}
We use an elementary approach with induction.  The case $\sigma = 0$ is implied by Corollary \ref{coro:EulerMacdiag}, so suppose $\sigma < 0$.  Note that
\begin{equation}
\label{eq:randomeq}
\sum_{m \leq y} \frac{1}{m^{1 + \sigma}} \leq 1 + \int_1^{y} t^{-1-\sigma} dt \leq 1 +  |\sigma|^{-1}y^{-\sigma}.
\end{equation}
This proves \eqref{eq:dkUB} for $k=1$.  Suppose $k \geq 2$.  Then using \eqref{eq:randomeq},
\begin{equation*}
\sum_{n \leq x} \frac{d_k(n)}{n} \leg{x}{n}^{\sigma} \leq x^{\sigma} \sum_{n \leq x} \frac{d_{k-1}(n)}{n^{1+\sigma}} \Big(1 + |\sigma|^{-1} \Big(\frac{x}{n}\Big)^{-\sigma} \Big).
\end{equation*}
Now apply the induction hypothesis 
to finish the proof.
\end{proof}

\subsection{A Mellin pair}
Recall $P_1[n]$ is given by \eqref{eq:P1def}.  Then for $n \leq y_1$, we have
\begin{equation}
\label{eq:P1Mellin}
 P_1[n] = \sum_{i} \frac{a_i}{(\log y_1)^i} (\log y_1/n)^i = \sum_{i} \frac{a_i i!}{(\log y_1)^i} \frac{1}{2 \pi i} \int_{(1)} \leg{y_1}{n}^s \frac{ds}{s^{i+1}}.
\end{equation}
For $n > y_1$, the right hand side vanishes and therefore agrees with $P_1[n]$ in this case also.

\section{The cross term}
\label{section:cross}
In this section we prove Lemma \ref{lemma:mainterm12}. 
\subsection{Averaging over $t$}
As a first step, we show
\begin{mylemma}
\label{lemma:I12intermediate}
Suppose $\theta_1, \theta_2$ satisfy
\begin{equation}
\label{eq:thetabounds}
\frac52 \theta_1 + \theta_2 < 2, \qquad \frac53 \theta_1 + \frac23 \theta_2 < 4/3, \qquad \theta_2 < \theta_1.
\end{equation}
 For any $B \geq 0$ we have, uniformly for $\alpha, \beta \ll L^{-1},$
\begin{equation}
\label{eq:I12intermediate}
 I_{12}(\alpha, \beta) =  \sum_{\substack{n \leq y_1; \; hk \leq y_2 \\ hl = nk} }  \frac{\mu(n)\mu_2(h) \sigma_{\alpha, -\beta}(l) }{(hkln)^{\half}} P_1[n] P_2[hk] e^{-l/T^3} \intR \Big(\frac{t}{2\pi} \Big)^{-\beta} w(t) dt + O_B(T/L^B).
\end{equation}
\end{mylemma}
\noindent {\bf Remark.}  When $\theta_1 = 4/7$ then \eqref{eq:thetabounds} translates to $\theta_2 < 4/7$.
\begin{proof}[Proof of Lemma \ref{lemma:I12intermediate}]
Inserting the definition of $\psi_1$ and $\psi_2$, we have
\begin{equation}
 I_{12}(\alpha,\beta) = \sum_{n \leq y_1} \sum_{hk \leq y_2} \frac{\mu(n)\mu_2(h) P_1[n] P_2[hk]}{(hkn)^{\half}} J_{12},
\end{equation}
where
\begin{equation}
 J_{12} = \intR w(t) \Big( \frac{h}{nk} \Big)^{-it} \chi(\thalf+it) \zeta(\thalf + \alpha + it) \zeta(\thalf + \beta -it)  dt.
\end{equation}
From the functional equation of $\zeta(\half + \beta -it)$ and the approximation
\begin{equation}
 \chi(\thalf + \beta -it) \chi(\thalf +it) = \Big(\frac{t}{2\pi} \Big)^{-\beta}(1 + O(t^{-1})),
\end{equation}
we get
\begin{equation}
\label{eq:J12middle}
 J_{12} = \intR w(t) \Big(\frac{t}{2\pi} \Big)^{-\beta} \Big( \frac{h}{nk} \Big)^{-it} \zeta(\thalf + \alpha + it) \zeta(\thalf - \beta +it)  dt + O(T^{\varepsilon}).
\end{equation}
Applying Lemma \ref{lemma:longAFE} to \eqref{eq:J12middle}, we have
\begin{equation}
 J_{12} = \sum_{l} \frac{\sigma_{\alpha, -\beta}(l)}{l^{\half}} e^{-l/T^3} \intR w(t) \Big(\frac{t}{2\pi} \Big)^{-\beta} \Big( \frac{hl}{nk} \Big)^{-it}  dt + O(T^{\varepsilon}).
\end{equation}
The $t$-integral takes the form $\widehat{w_0}(\frac{1}{2 \pi} \log\frac{hl}{nk})$ where $w_0(t) = w(t) \leg{t}{2\pi}^{-\beta}$.  
This explains why we took one long sum in the approximate functional equation above -- the sum over $l$ is automatically truncated by the decay of $\widehat{w_0}$.
The diagonal terms $hl = nk$ give the main term visible in \eqref{eq:I12intermediate}.  To complete the proof, we need to bound the off-diagonal terms (those with $hl \neq nk$).  We accomplish this with Lemma \ref{lemma:Cbound} below. \end{proof}

\subsection{Bounding the off-diagonal terms}
Let $C$ denote the contribution to $I_{12}(\alpha, \beta)$ from the off-diagonal terms, so that
\begin{equation}
\label{eq:Cdef}
 C = \sum_{\substack{n \leq y_1; \; hk \leq y_2 \\ hl \neq  nk} }  \frac{\mu(n)\mu_2(h) \sigma_{\alpha, -\beta}(l) }{(hkln)^{\half}} P_1[n] P_2[hk] e^{-l/T^3} \widehat{w_0}\Big(\frac{1}{2 \pi} \log\frac{hl}{nk}\Big).
\end{equation}
\begin{mylemma}
\label{lemma:Cbound} 
 Suppose $\theta_1, \theta_2$ satisfy \eqref{eq:thetabounds}.
Then for any $B \geq 0$ we have
\begin{equation}
 C \ll_B T/L^B.
\end{equation}
\end{mylemma}
\noindent The proof is fairly lengthy and we shall state and prove various intermediate lemmas within the body of the proof.

\begin{proof}
First note that an easy integration by parts argument shows 
\begin{equation}
\label{eq:w0bound}
\widehat{w_0}(x) \ll_B T (1 + \Delta x)^{-B}, 
\end{equation}
for any $B \geq 0$.  Similarly, for each $j=1, 2, \dots$, we have
\begin{equation}
\label{eq:w0'bound}
\frac{d^j}{dx^j} \widehat{w_0}(x) \ll_B T^{j+1} (1 + \Delta x)^{-B}.
\end{equation}

Define $f$ by $nk = hl-f$.  
Write $C = C' + C''$ where $C'$ corresponds to the terms with 
\begin{equation}
\label{eq:fsize}
|f| \leq \Delta^{-1 + \varepsilon} hl.
\end{equation}
  By \eqref{eq:w0bound}, we have $C'' \ll T^{-2009}$, by taking $B$ large enough with respect to $\varepsilon$, and bounding everything trivially.  Also note that a trivial bound on $C'$ gives
\begin{equation}
\label{eq:C'trivialbound}
 C' \ll T^{1+ \varepsilon} \sum_{\substack{n \leq y_1; \; hk \leq y_2 \\ hl =f+ nk \\ 0<|f| \ll \Delta^{-1+\varepsilon} y_1 y_2 } }  \frac{1 }{(hkln)^{\half}} \ll T^{\varepsilon} y_1 y_2,
\end{equation}
where the error term comes from letting $n$, $k$, and $f$ vary freely and bounding the number of values of $h$ and $l$ by the number of divisors of $f+nk$.  

Suppose now that \eqref{eq:fsize} holds.  A simple approximation of the logarithm gives
\begin{equation*}
 \widehat{w_0}\Big(\frac{1}{2 \pi} \log\frac{hl}{nk}\Big) = \widehat{w_0}\Big(\frac{f}{2 \pi hl } + O\big(\frac{f^2}{h^2 l^2}\big) \Big).
\end{equation*}
By the mean value theorem and \eqref{eq:w0'bound}, we then get (recall $\Delta = T/L$)
\begin{equation}
\label{eq:w0approx}
 \widehat{w_0}\Big(\frac{1}{2 \pi} \log\frac{hl}{nk}\Big) = \widehat{w_0}\Big(\frac{f}{2 \pi hl }\Big) + O(L^2). 
\end{equation}

Using \eqref{eq:w0approx} and the trivial bound used to prove \eqref{eq:C'trivialbound}, we then have
\begin{equation}
 C' = \sum_{\substack{n \leq y_1; \; hk \leq y_2 \\ hl = f+ nk \\ 0< |f| \leq \Delta^{-1 + \varepsilon} hl } }  \frac{\mu(n)\mu_2(h) \sigma_{\alpha, -\beta}(l) }{(hkln)^{\half}} P_1[n] P_2[hk] e^{-l/T^3} \widehat{w_0}\Big(\frac{f}{2 \pi hl }\Big) + O(T^{-1 + \varepsilon} y_1 y_2).
\end{equation}
Since $y_1 y_2 = T^{\theta_1 + \theta_2} \ll T^{2-\varepsilon}$ by \eqref{eq:thetabounds}, this error term is $O(T^{1-\varepsilon})$.

The next step is to eliminate $k$ and express the equation $hl = f + nk$ as the congruence $hl \equiv f \pmod{n}$.  Set $k = (hl-f)/n$ so that $k = \frac{hl}{n}(1+ O(\Delta^{-1 + \varepsilon}))$.  Note that
\begin{equation}
 P_2[hk] = P_2\Big[\frac{h^2l}{n} \Big(1- \frac{f}{hl}\Big)\Big] = P_2[h^2l/n] + O(\Delta^{-1 + \varepsilon}),
\end{equation}
which uses the fact that $\frac{d}{dx} P_2[x] \ll 1$ for $x \geq 1$ (recall that $P_2'(0) = 0$ so that $P_2[x]$ is continuously differentiable at $x=y_2$).  There is a small problem for $x < 1$ since $P_2[x]$ is really only defined for $x \geq 1$.  However, the terms with $h^2l/n \ll 1$ are not included in the sum since the sum over $f$ is empty unless $hl \geq \Delta^{1-\varepsilon}$, so that $h^2l/n \geq h \Delta^{1-\varepsilon}/n \geq \Delta^{1-\varepsilon}/y_1$.  The requirements \eqref{eq:thetabounds} imply that $\theta_1 < 1$ so that $h^2l/n \geq T^{\varepsilon}$.  Thus we obtain by a trivial estimation
\begin{equation}
 C = \sum_{\substack{n \leq y_1; \; h^2l/n \leq y_2; \; f\neq 0 \\ hl \equiv f \shortmod{n}}} 
\frac{\mu(n)\mu_2(h) \sigma_{\alpha, -\beta}(l) }{hl} P_1[n] P_2[h^2 l/n] e^{-l/T^3} \widehat{w_0}\Big(\frac{f}{2 \pi hl }\Big) + O(T^{1-\varepsilon}).
\end{equation}
We relaxed the condition $|f| \leq \Delta^{-1 + \varepsilon} hl$ without introducing a new error term due to the decay of $\widehat{w}_0$.

Now consider the inner sum over $l$ above, namely
\begin{equation*}
D(h,n, f):= \sum_{\substack{l \leq ny_2/h^2; \; f \neq 0 \\ hl \equiv f \shortmod{n}}} \frac{ \sigma_{\alpha, -\beta}(l) }{l} P_2[h^2 l/n] e^{-l/T^3} \widehat{w_0}\Big(\frac{f}{2 \pi hl }\Big),
\end{equation*}
so that
\begin{equation*}
 C = \sum_{n \leq y_1; \; h^2/n \leq y_2; \; f \neq 0} \frac{\mu(n) \mu_2(h) P_1[n]}{h} D(h,n,f) + O(T^{1-\varepsilon}).
\end{equation*}
The plan is to apply the Voronoi summation formula to $D(h,n,f)$.  In order to easily quote results from the literature, we shall treat the case $\alpha = \beta = 0$; versions of the Voronoi formula exist for general $\sigma_{\alpha, -\beta}(l)$ (e.g. see \cite{Motohashi}, Lemma 3.7).  Since $\alpha$ and $\beta$ are small, the methods will carry over essentially unchanged to the more general case.

\begin{mylemma}
\label{lemma:Voronoi}
Let $\eta(q)$ be the arithmetical function defined by $\eta(q) = \sum_{p|q} \frac{\log{p}}{p-1}$, where the sum is over primes, and let
\begin{equation}
\label{eq:gdef}
 g(l) = l^{-1} P_2[h^2l/n] e^{-l/T^3} \widehat{w_0}\Big(\frac{f}{2 \pi hl }\Big).
\end{equation}
Suppose that $(h, n) = b$, $n= b n_1$, $h= b h_1$, and $b | f$.  Then $D(h,n,f) = D_M + D_E$, where
\begin{equation}
\label{eq:DM}
 D_M = \frac{\phi(n_1)}{n_1^2} g(1) (2\gamma -1 - 2 \eta(n_1)) + \frac{\phi(n_1)}{n_1^2} \int_1^{x} g(t) (\log{t} + 2\gamma - 2 \eta(n_1)) dt,
\end{equation}
and for any $B \geq 0$, we have
\begin{equation}
\label{eq:DEbound}
 D_E \ll  T^{\varepsilon} ( \frac{\sqrt{n_1} h}{|f|} + \frac{T^{1/3} h^{2/3}}{|f|^{2/3}}) \Big(1 + \frac{|f|\Delta h}{n y_2}\Big)^{-B}.
\end{equation}
If $b \nmid f$ then the sum is void, i.e., $D(h,n,f) = 0$.
\end{mylemma}
\begin{proof}[Proof of Lemma \ref{lemma:Voronoi}]
Note that \eqref{eq:w0'bound} implies
\begin{equation}
\label{eq:g'bound}
 g'(t) \ll t^{-2} T^{1 + \varepsilon} \Big(1 + \frac{|f|\Delta}{ht}\Big)^{-B}.
\end{equation}
Then with $x = n y_2/h^2$, $f= b f_1$, we have
\begin{equation*}
 D(h,n,f) = \sum_{\substack{l \leq x; \; l \equiv \overline{h_1} f_1 \shortmod{n_1}}} d(l) g(l).
\end{equation*}
By partial summation, noting $g(x) = 0$, 
\begin{equation*}
 D(h,n,f) = - \int_1^{x} g'(t)  [\sum_{\substack{l \leq t; \; l \equiv \overline{h_1} f_1 \shortmod{n_1}}} d(l) ] dt.
\end{equation*}
The Voronoi formula (see Exercise 7 of \cite{IK} p.79, for example) gives that
\begin{equation}
\label{eq:Voronoi}
 \sum_{\substack{l \leq t; \; l \equiv \overline{h_1} f_1 \shortmod{n_1}}} d(l) = M + E,
\end{equation}
where
\begin{equation*}
 M = \frac{\phi(n_1)}{n_1^2} t (\log{t} + 2\gamma -1 - 2 \eta(n_1)), 
 \quad E \ll (n_1 t)^{\varepsilon} (\sqrt{n_1} + t^{1/3}).
\end{equation*}
Write $D(h,n,f) = D_M + D_E$ corresponding to \eqref{eq:Voronoi}.  That is,
\begin{equation*}
 D_M = - \frac{\phi(n_1)}{n_1^2} \int_1^{x} g'(t) t (\log{t} + 2\gamma -1 - 2 \eta(n_1)) dt.
\end{equation*}
Integration by parts gives \eqref{eq:DM}.

As for the error term, a straightforward computation with \eqref{eq:g'bound} gives
\begin{equation*}
 \int_1^{x} |g'(t)| dt \ll T^{\varepsilon} \frac{h}{|f|} \Big(1 + \frac{|f|\Delta}{hx}\Big)^{-B}, \qquad \int_1^{x} t^{1/3} |g'(t)| dt \ll \frac{T^{1/3 + \varepsilon} h^{2/3}}{|f|^{2/3}} \Big(1 + \frac{|f|\Delta}{hx}\Big)^{-B}.
\end{equation*}
Consequently, \eqref{eq:DEbound} follows.
\end{proof}

Now we return to the proof of Lemma \ref{lemma:Cbound}.  We show
\begin{mylemma}
\label{lemma:Cdecomp}
 Assume \eqref{eq:thetabounds} holds.  Then
\begin{equation}
 C = C_0 + O(T^{1-\varepsilon}),
\end{equation}
where
\begin{equation}
\label{eq:C0def}
 C_0  = \sum_{n \leq y_1} \sum_{h \leq n y_2/\Delta^{1-\varepsilon}} \sum_{ b|f; \; f \neq 0} \frac{\mu(n) \mu_2(h) P_1[n]}{h} \frac{\phi(n_1)}{n_1^2} \int_{\Delta^{1-\varepsilon}/h}^{n y_2/h^2} g(t) (\log{t} + 2\gamma - 2 \eta(n_1)) dt.
\end{equation}
\end{mylemma}
\begin{proof}[Proof of Lemma \ref{lemma:Cdecomp}]
Write $C = C_M + C_E + O(T^{1-\varepsilon})$ according to the decomposition $D(h,n,f) = D_M + D_E$, and similarly $C_M = C_1 + C_2$ according to the two terms of \eqref{eq:DM}.  A direct application of \eqref{eq:DEbound} shows
\begin{equation*}
 C_E \ll T^{\varepsilon} \frac{y_1^{5/2} y_2}{\Delta} + T^{\varepsilon} \frac{y_1^{5/3} y_2^{2/3}}{\Delta^{1/3}},
\end{equation*}
which is $\ll T^{1-\varepsilon}$ since \eqref{eq:thetabounds} holds.

Now we bound $C_1$.  The point is that essentially $f \ll h/\Delta$, yet $h \leq \sqrt{n y_2} = O(\Delta^{1-\varepsilon})$, so that the sum over $f$ practically has no length.  Explicitly, we have 
\begin{equation*}
C_1 \ll T^{\varepsilon} \sum_{n \leq y_1, h^2/n \leq y_2, f \neq 0} \frac{1}{n_1} \frac{1}{h} \Big|\widehat{w_0}\big(\frac{f}{2\pi h}\big) \Big| \ll T^{1+\varepsilon} \sum_{n \leq y_1} \sum_{h \leq \sqrt{y_1 y_2}} \frac{1}{n_1} \frac{1}{h} \leg{h}{\Delta}^B, 
\end{equation*}
which after writing $n = b n_1$, gives
\begin{equation*}
C_1 \ll T^{1+\varepsilon} \leg{\sqrt{y_1 y_2}}{T}^B \sum_{b \leq y_1} \sum_{n_1 \leq y_1/b} \frac{1}{n_1} \ll T^{2+\varepsilon} \leg{\sqrt{y_1 y_2}}{T}^B.
\end{equation*}
Since \eqref{eq:thetabounds} holds, we have $\theta_1 + \theta_2 < 2$.  Taking $B = 2(1-\frac{\theta_1+\theta_2}{2})^{-1}$ gives $C_1 \ll T^{\varepsilon}$.

By definition,
\begin{equation*}
 C_2  = \sum_{n \leq y_1} \sum_{h \leq \sqrt{ny_2}} \sum_{f \neq 0} \frac{\mu(n) \mu_2(h) P_1[n]}{h} \frac{\phi(n_1)}{n_1^2} \int_{1}^{n y_2/h^2} g(t) (\log{t} + 2\gamma - 2 \eta(n_1)) dt,
\end{equation*}
so that $C_0$ is given by the same expression except the integral has lower bound at $\Delta^{1-\varepsilon}/\sqrt{y_1 y_2}$.  Again, the point is that essentially $1 \leq |f| \ll ht/\Delta^{1-\varepsilon}$.  An argument similar to that used above in bounding $C_1$ shows that $C_2 = C_0 + O(T^{\varepsilon})$.
\end{proof}

Now the proof of Lemma \ref{lemma:Cbound} is reduced to showing $C_0 \ll T/L^B$. We arrange $C_0$ as follows
\begin{equation}
C_0 = \sum_{h \leq y_1 y_2 /\Delta^{1-\varepsilon}} \frac{\mu_2(h)}{h} \int_{\Delta^{1-\varepsilon}/h}^{y_1 y_2/h^2} t^{-1} e^{-t/T^3} S_1 S_2 dt,
\end{equation}
where
\begin{equation*}
S_1 = \sum_{h^2 t/y_2 \leq n \leq y_1} \mu(n) P_1[n] P_2[h^2t/n] \frac{\phi(n_1)}{n_1^2} (\log{t} + 2\gamma - 2 \eta(n_1)),
\end{equation*}
and
\begin{equation*}
S_2 = \sum_{b|f; \; f \neq 0} \widehat{w_0}(f/2\pi ht).
\end{equation*}
\begin{mylemma}
\label{lemma:S2}
We have
\begin{equation}
\label{eq:S2bound}
 S_2 \ll T,
\end{equation}
uniformly in $h$ and $t$.
\end{mylemma}
\begin{mylemma}
\label{lemma:S1}
 For any $B \geq 0$, we have
\begin{equation}
\label{eq:S1bound}
 S_1 \ll_B L^{-B}.
\end{equation}
\end{mylemma}

\noindent An easy application of Lemmas \ref{lemma:S2} and \ref{lemma:S1} finally completes the proof of Lemma \ref{lemma:Cbound}, since they show $C_0 \ll T/L^{B}$.
\end{proof}

\begin{proof}[Proof of Lemma \ref{lemma:S2}]
Recall that $b = (h, n)$, so that $S_2$ takes the form
\begin{equation*}
 \sum_{f_1 \neq 0} \widehat{w_0}(f_1/2 \pi h_1 t).
\end{equation*}
Let $X \geq 1$ be a parameter.  Then by Poisson summation,
\begin{align*}
 \sum_{q \neq 0} \widehat{w_0}(q/X) &= - \widehat{w_0}(0) + \sum_{k \in \mz} \intR \widehat{w_0}(u/X) e(-uk) du
\\
&= -\widehat{w_0}(0) + X \sum_{k \in \mz} w_0(kX).
\end{align*}
Since $w$ has support in $[T/4, 2T]$, we have $\widehat{w_0}(0) \ll T$ (recall $w_0(x) = w(x) (x/2\pi)^{-\beta}$ and that $\beta \ll 1/\log{T}$).  The support of $w_0$ also implies that $k \asymp T/X$ so that the sum over $k$ also gives $O(T)$.  Note that if $X > 2T$ then the sum over $k$ is identically zero.  
\end{proof}

\begin{proof}[Proof of Lemma \ref{lemma:S1}]
This is essentially a variation on the prime number theorem.

To begin, write $n = bn_1$ where $(n_1, h_1) = 1$.  Then
\begin{equation*}
 S_1 = \sum_{\substack{b h_1^2 t/y_2 \leq  n_1 \leq y_1/b \\ (n_1, h_1)=1}} \mu(bn_1) P_1[b n_1] P_2[b h_1^2t/n_1] \frac{\phi(n_1)}{n_1^2} (\log{t} + 2\gamma - 2 \eta(n_1)).
\end{equation*}
For simplicity, consider
\begin{equation*}
 S_1' := \sum_{n} \mu(n) P_1[n] P_2[R/n] \frac{1}{n},
\end{equation*}
for a parameter $R \geq y_1 T^{\varepsilon}$ (in our application, $R = bh_1^2 t \geq \Delta^{1-\varepsilon} \geq y_1 T^{\varepsilon}$). We will show
\begin{equation}
\label{eq:S1'bound}
 S_1' \ll L^{-B}.
\end{equation}
A straightforward modification of the method gives the same bound for $S_1$.  

By convention, $P_1[n]$ is zero unless $1 \leq n \leq y_1$, and similarly $P_2[R/n] = 0$ unless $R/y_2 \leq n \leq R$.  The condition $n \leq R$ is implied by the condition $n \leq y_1$.  By taking Mellin transforms of $P_1$ and $P_2$ (that is, using \eqref{eq:P1Mellin}), we get
\begin{equation*}
 S_1' = \sum_j \sum_k \frac{a_j j! b_k k!}{(\log y_1)^j (\log y_2)^k} \sum_{n} \Big(\frac{1}{2 \pi i} \Big)^2 \int_{(\varepsilon)} \int_{(3\varepsilon)} y_1^s (y_2/R)^u \frac{\mu(n)}{n^{1+s-u}} \frac{ds}{s^{j+1}} \frac{du}{u^{k+1}}.
\end{equation*}
The sum over $n$ converges absolutely now, so that
\begin{equation*}
 S_1' = \sum_j \sum_k \frac{a_j j! b_k k!}{(\log y_1)^j (\log y_2)^k} \Big(\frac{1}{2 \pi i} \Big)^2 \int_{(\varepsilon)} \int_{(3\varepsilon)} y_1^s (y_2/R)^u \frac{1}{\zeta(1+s-u)} \frac{ds}{s^{j+1}} \frac{du}{u^{k+1}}.
\end{equation*}
Now change variables $s \rightarrow s + u$, so that with
\begin{equation}
\label{eq:Fdef}
 F(s) = \frac{1}{2 \pi i} \int_{(2\varepsilon)} x^u \frac{du}{(u+s)^{j+1} u^{k+1}}, \quad x=\frac{y_1 y_2}{R},
\end{equation}
we have
\begin{equation}
\label{eq:S1'F}
 S_1' = \sum_j \sum_k \frac{a_j j! b_k k!}{(\log y_1)^j (\log y_2)^k} \frac{1}{2 \pi i} \int_{(\varepsilon)} y_1^s \frac{1}{\zeta(1+s)} F(s) ds.
\end{equation}
In the integral representation \eqref{eq:Fdef} for $F(s)$, we initially suppose $\text{Re}(s) < 2\varepsilon$.

We need to develop the analytic properties of $F$.  It turns out that $F(s)$ is an entire function of $s$ and has good growth properties. 
First, if $x \leq 1$ then by moving the contour far to the right we get the uniform  bound $F(s) \ll (1+|s|)^{-j-1}$.  Even when $x  > 1$ we can do this, which shows that $F$ is entire.  Next suppose that $x > 1$.  
In this case, we show that 
if $|s| > (\log{x})^{-1}$ we have a formula for $F$ of the form
\begin{equation}
\label{eq:Fformula}
F(s) = \sum_{l \leq k} c_{j,k,l} \frac{(\log{x})^{k-l}}{s^{j+l+1}} + x^{-s} \sum_{l \leq j} d _{j,k,l} \frac{(\log{x})^{j-l}}{s^{k+l+1}} + O((1+|s|)^{-j-1}),
\end{equation}
for certain constants $c_{j,k,l}$, $d_{j,k,l}$ (we will give them explicitly below).
With the 
formula \eqref{eq:Fformula} it is straightforward to prove \eqref{eq:S1'bound} as in Chapter 18 of \cite{Davenport} for example. It is a key point that $y_1^s \leg{y_1 y_2}{R}^{-s} = \leg{R}{y_2}^s$ becomes small for $s$ with negative real part.

We prove \eqref{eq:Fformula} now.  Since $x>1$, we move the contour far to the left, crossing poles as $u=0, u=-s$.  The new contour is again bounded by $(1+|s|)^{-j-1}$, accounting for the error term in \eqref{eq:Fformula}.  The residues can be expressed as contour integrals of the form
\begin{equation}
 \label{eq:Fcontour}
 \frac{1}{2\pi i} \oint x^u \frac{du}{(u+s)^{j+1} u^{k+1}},
 \end{equation}
 with one a small circle around $u=0$, and the other around $u=-s$ (if $s=0$ then it is just one contour integral around $u=0$). 
Consider \eqref{eq:Fcontour} on the circle of very small radius (small compared to $|s|$) around $u=0$.  By the binomial theorem,
\begin{equation*}
\frac{1}{(u+s)^{j+1}} = \frac{1}{s^{j+1}} \sum_{l=0}^{\infty} (-1)^l \binom{j+l}{j} \leg{u}{s}^l.
\end{equation*}
Inserting this into \eqref{eq:Fcontour} and reversing the order of summation and integration gives
\begin{equation*}
\frac{1}{s^{j+1}} \sum_{l=0}^{\infty} (-1)^l \binom{j+l}{j} s^{-l} \frac{1}{2\pi i} \oint x^u \frac{du}{u^{k-l+1}} = \sum_{l \leq k} (-1)^l \binom{j+l}{j} \frac{(\log{x})^{k-l}}{s^{j+l+1}(k-l)!}.
\end{equation*}
This accounts for the first term in \eqref{eq:Fformula} above.
The second contour around $u=-s$ can be reduced to an instance of this previous formula after changing variables $u \rightarrow u-s$.  This switches the roles of $j$ and $k$, and multiplies by $x^{-s}$.  
\end{proof}

\subsection{Reduction to a contour integral}
Next we express the main term in \eqref{eq:I12intermediate} as a contour integral.  By \eqref{eq:P1Mellin} (and similarly for $P_2$), we have
\begin{multline}
 I_{12}(\alpha, \beta) = \widehat{w_0}(0) \sum_{i,j} \frac{a_i b_j i! j!}{(\log y_1)^i (\log y_2)^j}
\leg{1}{2\pi i}^3 \int_{(1)} \int_{(1)} \int_{(1)}  T^{3z} \Gamma(z) y_1^s y_2^u
\\
\sum_{\substack{hl = nk} }  \frac{\mu(n)\mu_2(h) \sigma_{\alpha, -\beta}(l) }{(hk)^{\half + u} n^{\half + s} l^{\half +z}} \frac{dz \, ds \, du}{s^{i+1} u^{j+1}} + O(T^{1-\varepsilon}).
\end{multline}
A calculation gives
\begin{equation}
\label{eq:arithmeticalcross}
 \sum_{hl = nk}  \frac{\mu(n)\mu_2(h) \sigma_{\alpha, -\beta}(l) }{(hk)^{\half + u} n^{\half + s} l^{\half +z}} = \frac{\zeta^2(1+u+s) \zeta(1 + \alpha + u + z) \zeta(1 - \beta + u + z)}{\zeta^2(1 + 2u) \zeta(1 + \alpha + s + z)\zeta(1 - \beta + s + z)} A(s,u,z),
\end{equation}
where $A(s,u,z)$ is a certain arithmetical factor that is given by an Euler product that is absolutely and uniformly convergent in some product of fixed half-planes containing the origin.  We first move the $s$ and $u$ contours to $\text{Re} = \delta$, and then move the $z$-contour to $-2\delta/3$, where $\delta > 0$ is some fixed constant such that the arithmetical factor converges absolutely.  By doing so we only cross a pole at $z=0$.  On the new line we simply bound the integral by absolute values, giving the following contribution to $I_{12}$
\begin{equation*}
 \ll |\widehat{w_0}(0)| \leg{y_1 y_2}{T^{2}}^{\delta} \ll T^{1-\varepsilon},
\end{equation*}
since $\theta_1 + \theta_2 < 2$.  Thus
\begin{equation}
\label{eq:I12}
 I_{12}(\alpha,\beta) = \widehat{w_0}(0) \sum_{i,j} \frac{a_i b_j i! j!}{(\log y_1)^i (\log y_2)^j}
 K_{12} + O(T^{1-\varepsilon}),
\end{equation}
where
\begin{equation}
K_{12} =  \leg{1}{2\pi i}^2 \int_{(\varepsilon)} \int_{(\varepsilon)} y_1^s y_2^u \frac{\zeta^2(1+u+s) \zeta(1 + \alpha + u) \zeta(1 - \beta + u)}{\zeta^2(1 + 2u) \zeta(1 + \alpha + s)\zeta(1 - \beta + s)} A(s,u,0) \frac{ds du}{s^{i+1} u^{j+1}}.
\end{equation}

\subsection{Evaluation of $K_{12}$}
We now evaluate $K_{12}$ asymptotically; this is somewhat subtle.  Considering the pole of $\zeta^2(1+u+s)$ at $u=-s$, we are wary about moving contours into the critical strip.  To get around this delicate issue, we shall separate the variables $s$ and $u$.

First notice that by moving the contours of integration to $\varepsilon \asymp 1/L$, and bounding the integral with absolute values, we see that $K_{12} \ll L^{i+j}$ (which translates to showing $I_{12}(\alpha,\beta)$ is asymptotically constant as $T \rightarrow \infty$, consistent with successful mollification).  
Let $K_{12}'$ be the same integral as $K_{12}$ but with $A(s,u,0)$ replaced by $A(0,0,0)$.  We check in Section \ref{section:crossarithmeticalfactor} below that $A(0,0,0) = 1$, a result we now use freely.  Then we see that $K_{12} = K_{12}' + O(L^{i+j-1})$.  

Next we replace $\zeta^2(1+s+u)$ by its Dirichlet series and reverse the orders of summation and integration.  This cleanly separates the variables $s$ and $u$.  Thus we get
\begin{equation}
\label{eq:K12'}
K_{12}' = \sum_{n \leq y_2} \frac{d(n)}{n} K_1 K_2
,
\end{equation}
where
\begin{align}
K_1 &= \frac{1}{2\pi i} \int_{(\varepsilon)}  \leg{y_1}{n}^s  \frac{1}{\zeta(1 + \alpha + s)\zeta(1 - \beta + s)} \frac{ds }{s^{i+1}} 
\\
K_2 &= \frac{1}{2\pi i} \int_{(\varepsilon)} \leg{y_2}{n}^u \frac{\zeta(1 + \alpha + u) \zeta(1 - \beta + u)}{\zeta^2(1 + 2u) } \frac{du}{u^{j+1}}.
\end{align}
Here we were able to truncate $n$ at $y_2 < y_1$ by moving the $u$-integral far to the right.  

We compute the $s$ and $u$ integrals separately with the following
\begin{mylemma}
\label{lemma:K1K2}
Suppose $i \geq 1$, $j \geq 3$.  Then we have
\begin{equation}
\label{eq:K1calc}
K_1 = 
\frac{1}{i!} \frac{d^2}{dx dy} \left[e^{\alpha x - \beta y} (x + y + \log{y_1/n})^i \right]_{x=y=0} + O(L^{i-3}),
\end{equation}
and
\begin{equation}
 \label{eq:K2j2}
  K_2 = \frac{4 (\log (y_2/n))^{j}}{(j-2)!} \mathop{\int \int}_{\substack{a+b \leq 1\\ 0\leq a,b}} (1-a-b)^{j-2} \leg{y_2}{n}^{-a\alpha  + b\beta }   da db + O(L^{j-1}).
 \end{equation}
\end{mylemma}

\begin{proof}[Proof of Lemma \ref{lemma:K1K2}]
We first work on $K_1$.  An argument on the level of the prime number theorem shows that $K_1$ is captured by the residue at $s=0$, with an error of size $(\log y_1/n)^{-A}$ for arbitrarily large $A$.  Since $n \leq y_2$, we have $\log y_1/n \geq \log y_1/y_2 = (\theta_1 - \theta_2) L$ so that this error is satisfactory (if we had $y_1 = y_2$ the error term would be more delicate; see Lemma \ref{lemma:L1computed}).  A similar approximation argument shows that
\begin{equation*}
K_1 = \frac{1}{2 \pi i} \oint \leg{y_1}{n}^s  (\alpha + s)(-\beta + s)\frac{ds }{s^{i+1}} +O(L^{i-3}),
\end{equation*}
where the contour is a small circle enclosing $0$.
We calculate this integral exactly as
\begin{equation*}
K_1 = (-\alpha \beta) \frac{(\log (y_1/n))^{i}}{i!} + (\alpha-\beta) \frac{(\log (y_1/n))^{i-1}}{(i-1)!} + \frac{(\log (y_1/n))^{i-2}}{(i-2)!} + O(L^{i-3}).
\end{equation*}
When $i=1$ we interpret $1/(i-2)! = 0$.
This can be expressed in a compact form as \eqref{eq:K1calc}.

Now we compute $K_2$.  As in the computation of $K_1$, the prime number theorem shows that we can replace the contour by a small circle around the origin with radius $\asymp L^{-1}$, with error $O(1)$.  Then on this contour 
\begin{equation}
\label{eq:K2circle}
K_2 = 4 \frac{1}{2\pi i} \oint q^u \frac{1}{(\alpha+u)(-\beta+u)}  \frac{du}{u^{j-1}} + O(L^{j-1}), 
\end{equation}
where $q= y_2/n$.  
Observe the following identity
\begin{equation}
\label{eq:reciprocal}
 \frac{1}{\alpha +u} = \int_{1/q}^{1} r^{\alpha+u-1} dr + \frac{q^{-\alpha-u}}{\alpha+u},
\end{equation}
valid for all complex numbers $\alpha + u$ and positive $q$.  We apply this identity to $K_2$, expressing the main term as the sum of these two terms.  The latter term can be seen to give no contribution to $K_2$: it is
\begin{equation*}
 4 q^{-\alpha} \frac{1}{2\pi i} \oint \frac{1}{(\alpha + u)(-\beta+u)}  \frac{du}{u^{j-1}},
\end{equation*}
which vanishes as can be seen by taking the contour to be arbitrarily large.  By reversing the orders of integration we have
\begin{equation*}
 K_2 = 4 \int_{1/q}^{1} r^{\alpha-1} \frac{1}{2 \pi i} \oint (qr)^{u} \frac{1}{-\beta + u} \frac{du}{u^{j-1}} dr + O(L^{j-1}).
\end{equation*}
We use \eqref{eq:reciprocal} again but with the lower boundary of integration at $1/qr$.  Again we get $K_2$ as the sum of two terms, with the latter term vanishing (using $j \geq 2$).  
The former term is
\begin{equation*}
 4 \int_{1/q}^{1} \int_{1/qr}^{1} r^{\alpha-1} t^{-\beta-1} \frac{1}{2\pi i} \oint (qrt)^{u} \frac{du}{u^{j-1}} dt dr,
\end{equation*}
which 
can be calculated as
\begin{equation*}
 \frac{4}{(j-2)!} \int_{1/q}^{1} \int_{1/qr}^{1} r^{\alpha-1} t^{-\beta-1} (\log rt \frac{y_2}{n})^{j-2} dt dr.
\end{equation*}
Changing variables by $r = q^{-a}, t = q^{-b}$ and simplifying gives \eqref{eq:K2j2}.
\end{proof}

\subsection{Final simplifications}
We are finally ready to finish the proof of Lemma \ref{lemma:mainterm12}.  
\begin{proof}
We pick up our calculation with \eqref{eq:I12} and \eqref{eq:K12'},
getting
\begin{equation*}
 I_{12}(\alpha,\beta) = \widehat{w_0}(0) \sum_{n \leq y_2} \frac{d(n)}{n} \Big(\sum_{i} \frac{a_i i!}{(\log y_1)^i} K_1 \Big) \Big(\sum_{j} \frac{b_j j!}{(\log y_2)^j} K_2 \Big) + O(T/L).
\end{equation*}
Using Lemma \ref{lemma:K1K2} we now compute these sums over $i$ and $j$.  We have
\begin{equation*}
 \sum_i = \frac{d^2}{dx dy}\Big[ e^{\alpha x - \beta y} P_1\Big(\frac{x+y}{\log y_1} + \frac{\log y_1/n}{\log y_1} \Big) \Big]_{x=y=0} 
+ O(L^{-3}),
\end{equation*}
which we write in the slightly more convenient form as
\begin{equation*}
\sum_i= \frac{1}{(\log y_1)^2} \frac{d^2}{dx dy}\Big[ y_1^{\alpha x - \beta y} P_1\Big(x+y + \frac{\log y_1/n}{\log y_1} \Big) \Big]_{x=y=0} + O(L^{-3}).
\end{equation*}
The sum over $j$ is
\begin{equation*}
 \sum_j = 4 \frac{(\log y_2/n)^{2}}{(\log y_2)^{2}} \mathop{\int \int}_{0 \leq a+ b \leq 1}  \leg{y_2}{n}^{-a\alpha + b\beta} P_2''\Big((1-a-b) \frac{\log (y_2/n)}{\log y_2} \Big) da db
+ O(L^{-1}).
\end{equation*}
With this, and recalling $\widehat{w_0}(0) = T^{-\beta} \widehat{w}(0)(1+O(L^{-1}))$, we have
\begin{multline}
 I_{12}(\alpha,\beta) =  \frac{4 T^{-\beta} \widehat{w}(0)}{(\log y_1)^2} \frac{d^2}{dx dy}\Big[ y_1^{\alpha x - \beta y} \mathop{\int \int}_{0 \leq a+b \leq 1} 
\sum_{n \leq y_2} \frac{d(n)}{n}  \frac{(\log (y_2/n))^{2}}{(\log y_2)^{2}}
\\
\leg{y_2}{n}^{-a\alpha + b\beta} P_1\Big(x+y + \frac{\log (y_1/n)}{\log y_1} \Big) P_2''\Big((1-a-b) \frac{\log (y_2/n)}{\log y_2} \Big) \Big]_{x=y=0} da db   + O(T/L).  
\end{multline}
Finally we apply Lemma \ref{lemma:EulerMacCross} to the sum over $n$ (with $k=2$, $x=y_1$, $z=y_2$, $s=-a\alpha +  b\beta$, $F(u) = P_1(x+y+u)$, $H(u) = u^2 P_2''((1-a-b)u)$) to finish the proof of Lemma \ref{lemma:mainterm12}.
\end{proof}

\subsection{The arithmetical factor}
\label{section:crossarithmeticalfactor}
Here we verify that $A(0,0,0) = 1$.  To do so, we compute $A(s,s,s)$ and verify it is $1$ at $s=0$.  In view of \eqref{eq:arithmeticalcross}, we have
\begin{equation*}
 A(s,s,s) = \sum_{hl = nk}  \frac{\mu(n)\mu_2(h) \sigma_{\alpha, -\beta}(l) }{(hknl)^{\half + s} } = \sum_{hl = nk}  \frac{\mu(n)\mu_2(h) \sigma_{\alpha, -\beta}(l) }{(hl)^{1 + 2s} }.
\end{equation*}
For $hl$ fixed, the sum over $nk$ is $\sum_{n | hl} \mu(n)$ which picks out $h=l=1$.  Thus $A(s,s,s) = 1$ for all $s$!

\section{The second diagonal term}
\label{section:diag}
Our goal in this section is to prove Theorem \ref{thm:mainterm2}.  The overall strategy is roughly the same as that in Section \ref{section:cross}.
\subsection{Reduction to a contour integral}
Recall $I_2(\alpha, \beta)$ is defined by \eqref{eq:I2def}.  Writing out the definition of $\psi_2$, we have
\begin{equation}
 I_{2}(\alpha,\beta) = 
\sum_{\substack{h_1, h_2 \\ k_1, k_2}} 
\frac{\mu_2(h_1) \mu_2(h_2) P_2[h_1 k_1] P_2[h_2 k_2]}{\sqrt{h_1 h_2 k_1 k_2}} \intR \left(\tfrac{h_1 k_2}{h_2 k_1}\right)^{-it} w(t) \zeta(\thalf + \alpha + it) \zeta(\thalf + \beta - it) dt.
\end{equation}
We apply Proposition \ref{prop:twisted}, writing $I_2(\alpha, \beta) = I_2'(\alpha, \beta) + I_2''(\alpha, \beta)$, where $I_2''$ can be obtained from $I_2'$ by switching $\alpha$ and $-\beta$ and multiplying by $\leg{t}{2\pi}^{-\alpha-\beta} = T^{-\alpha - \beta} + O(L^{-1})$.  Thus
\begin{multline}
 I_{2}'(\alpha,\beta) = \intR w(t) \sum_{i,j} \frac{b_i b_j i! j!}{(\log{y_2})^{i+j}} \sum_{h_1 k_2 m = h_2 k_1 n} \frac{\mu_2(h_1) \mu_2(h_2) }{(h_1 h_2 k_1 k_2)^{\half} m^{\half + \alpha} n^{\half + \beta}} 
\\
\leg{1}{2 \pi i}^3  \int_{(1)} \int_{(1)}\int_{(1)}\leg{y_2}{h_1 k_1}^s \leg{y_2}{h_2 k_2}^u \leg{t}{2 \pi mn}^{z} \frac{G(z)}{z} dz \frac{ds}{s^{i+1}} \frac{du}{u^{j+1}} dt.
\end{multline}
Next we compute the arithmetical sum as
\begin{multline}
\label{eq:B}
 \sum_{h_1 k_2 m = h_2 k_1 n} \frac{\mu_2(h_1) \mu_2(h_2) }{(h_1 k_1 )^{\half+s} (h_2 k_2)^{\half+u} m^{\half + \alpha+z} n^{\half + \beta+z}} 
\\
= \frac{\zeta(1+s+u)^5 \zeta(1+ \alpha + s + z) \zeta(1 + \beta + u + z) \zeta(1 + \alpha + \beta + 2z)}{\zeta^2(1+2s) \zeta^2(1+2u) \zeta^2(1+\beta + s + z)\zeta^2(1+\alpha + u + z)} B(s,u,z),
\end{multline}
where $B(s,u,z)$ is an arithmetical factor converging absolutely in a product of half-planes containing the origin.  Hence
\begin{multline*}
 I_2'(\alpha,\beta) = \intR w(t) \sum_{i,j} \frac{b_i b_j i! j!}{(\log{y_2})^{i+j}} 
\leg{1}{2 \pi i}^3  \int_{(1)} \int_{(1)}\int_{(1)} B(s,u,z)y_2^{s+u} \leg{t}{2 \pi}^{z} \frac{G(z)}{z}
\\
\frac{\zeta(1+s+u)^5 \zeta(1+ \alpha + s + z) \zeta(1 + \beta + u + z) \zeta(1 + \alpha + \beta + 2z)}{\zeta^2(1+2s) \zeta^2(1+2u) \zeta^2(1+\beta + s + z)\zeta^2(1+\alpha + u + z)}  dz \frac{ds}{s^{i+1}} \frac{du}{u^{j+1}} dt.
\end{multline*}
Now we take the contours of integration to $\delta > 0$ small, and then move $z$ to $-\delta + \varepsilon$, crossing a simple pole at $z=0$ only (since $G(z)$ vanishes at the pole of $\zeta(1 + \alpha + \beta + 2z)$).  The new line of integration contributes
\begin{equation}
 \ll T^{1+\varepsilon} \leg{y_2^2}{T}^{ \delta} = O(T^{1-\varepsilon}),
\end{equation}
since $\theta_2 < 1/2$.  Write $I_2'(\alpha, \beta) = I_{20}'(\alpha + \beta) + O(T^{1-\varepsilon})$, where $I_{20}'(\alpha, \beta)$ corresponds to the residue at $z=0$. Then
\begin{equation}
\label{eq:I2'}
I_{20}'(\alpha, \beta)= \widehat{w}(0) \zeta(1 + \alpha + \beta) \sum_{i,j} \frac{b_i b_j i! j!}{(\log{y_2})^{i+j}} J_2,
\end{equation}
where
\begin{equation*}
 J_2 = \leg{1}{2 \pi i}^2  \int_{(\delta)}\int_{(\delta)} y_2^{s+u} \frac{B(s,u,0) \zeta(1+s+u)^5 \zeta(1+ \alpha + s) \zeta(1 + \beta + u)}{\zeta^2(1+2s) \zeta^2(1+2u) \zeta^2(1+\beta + s)\zeta^2(1+\alpha + u)}  \frac{ds}{s^{i+1}} \frac{du}{u^{j+1}}.
\end{equation*}
Using the Dirichlet series for $\zeta^5(1+s+u)$ and reversing the order of the sum and integral, we get
\begin{multline*}
 J_{2} = \sum_{m \leq y_2} \frac{d_5(m)}{m} \leg{1}{2 \pi i}^2  \int_{(\delta)}\int_{(\delta)} B(s,u,0) \leg{y_2}{m}^{s+u} 
\\
\frac{\zeta(1+ \alpha + s) \zeta(1 + \beta + u)}{\zeta^2(1+2s) \zeta^2(1+2u) \zeta^2(1+\beta + s)\zeta^2(1+\alpha + u)}  \frac{ds}{s^{i+1}} \frac{du}{u^{j+1}}.
\end{multline*}
Taking $\delta \asymp L^{-1}$ and bounding the integrals trivially shows $J_2 \ll L^{i+j-1}$.  In particular, we can use a Taylor series so that $B(s,u,0) = B(0,0,0) + O(|s| + |u|)$ to write $J_2 = J_2' + O(L^{i+j-2})$, say.
Now the variables are separated so that
\begin{equation}
\label{eq:J2'}
J_2' = \sum_{m \leq y_2} \frac{d_5(m)}{m} L_1 L_2,
\end{equation}
where
\begin{equation}
\label{eq:L1}
L_1 = \frac{1}{2 \pi i} \int_{(\delta)} \leg{y_2}{m}^{s} \frac{\zeta(1+ \alpha + s) }{\zeta^2(1+2s) \zeta^2(1+\beta + s)}  \frac{ds}{s^{i+1}}, 
\end{equation}
and $L_2$ is the same as $L_1$ but with $i$ replaced by $j$ and $\alpha$ and $\beta$ switched.
Next we need to use the zero-free region for $\zeta$ to move the contour to the left of $0$.  Unfortunately, an error of size $O(1)$ is not sufficient for our application so we need a more subtle argument.  We have
\begin{mylemma}
\label{lemma:L1computed}
With $L_1$ defined by \eqref{eq:L1} and for some $\nu \asymp (\log \log y_2)^{-1}$ we have
\begin{equation}
\label{eq:L1circleint}
L_1 = 4 \frac{1}{2 \pi i} \oint\leg{y_2}{m}^s \frac{(\beta + s)^2}{\alpha + s} \frac{ds}{s^{i-1}} + O(L^{i-4}) + O\Big(\big(\frac{y_2}{m}\big)^{-\nu} L^{\varepsilon}\Big),
\end{equation}
where the contour is a circle of radius one enclosing the origin.
\end{mylemma}
\begin{proof}

Let $Y = o(T)$ be a large parameter to be chosen later.  By Cauchy's theorem, $L_1$ is equal to the sum of residues at $s=0$ and $s=-\alpha$, 
plus integrals over the line segments $\gamma_1 = \{ s = it, t \in \mr, |t| \geq Y \}$, $\gamma_{2} = \{s = \sigma \pm iY, -\frac{c}{\log{Y}} \leq \sigma \leq 0\}$, and $\gamma_3 = \{s=-\frac{c}{\log{Y}} + it, |t| \leq Y\}$, where $c$ is some fixed positive constant such that $\zeta(1+2s) \zeta(1 + \beta + s)$ has no zeros in the region on the right hand side of the contour determined by the $\gamma_i$.  Furthermore, we require that for such $c$ that $1/\zeta(\sigma + it) \ll \log(2 + |t|)$ in this region (see \cite{T} Theorem 3.11).  Then the integral over $\gamma_1$ is $\ll (\log{Y})^3/Y^{i} \ll Y^{-2}$ since $i \geq 3$.  The integral over $\gamma_2$ is $\ll (\log{Y})^2/Y^{i+1} \ll Y^{-2}$.  Finally, the contribution from $\gamma_3$ is $\ll (\log{Y})^i \leg{y_2}{m}^{-c/\log{Y}}$.  Choosing $Y \asymp \log y_2$ gives an error so far of size $O((y_2/m)^{-\nu} L^{\varepsilon}) + O(L^{-2})$.

Next we work with the sum of residues which can be expressed as
\begin{equation*}
\frac{1}{2 \pi i} \oint \leg{y_2}{m}^s \frac{\zeta(1+\alpha+s)}{\zeta^2(1+2s) \zeta^2(1+\beta + s)}  \frac{ds}{s^{i+1}},
\end{equation*}
where the contour is a circle of radius $\asymp 1/L$.
This integral is trivially bounded by $O(L^{i-3})$ so that taking the first term in the Taylor series of the $\zeta$'s finishes the proof.
\end{proof}
Next we calculate the integral in \eqref{eq:L1circleint} with the following
\begin{mylemma}
\label{lemma:L1exact}
For $i \geq 3$ we have
\begin{equation*}
\frac{1}{2 \pi i} \oint\leg{y_2}{m}^s \frac{(\beta + s)^2}{\alpha + s} \frac{ds}{s^{i-1}} = 
 \frac{d^2}{dx^2} \frac{(x+\log\frac{y_2}{m})^{i-1}}{(i-2)!} \int_0^{1} (1-u)^{i-2} e^{x(\beta - \alpha u)} \leg{y_2}{m}^{-\alpha u} du \Big|_{x=0}.
\end{equation*}
\end{mylemma}
\begin{proof}
Let $N$ be the integral to be computed.
We begin with the identity
\begin{equation*}
 (\beta+s)^2 =  \frac{d^2}{dx^2} e^{(\beta+s)x} \Big|_{x=0},
\end{equation*}
whence
\begin{equation*}
N= \frac{d^2}{dx^2} \left. e^{\beta x} N_1(x) \right|_{x=0}, \quad \text{where} \quad N_1(x) = \frac{1}{2 \pi i} \oint \left( e^{x} \frac{y_2}{m}\right)^{s} \frac{1}{\alpha + s} \frac{ds}{s^{i-1}}. 
\end{equation*}
Taking a power series, we have
\begin{equation*}
N_1(x) = \sum_{l \geq 0} \frac{(x + \log\frac{y_2}{m})^l}{l!} \frac{1}{2 \pi i} \oint \frac{s^{l-i+1}}{\alpha + s} ds.
\end{equation*}
As there are two poles inside the contour, it is easier to compute the residue at infinity.  In other words, change variables $s \rightarrow 1/s$ to get
\begin{equation*}
N_1(x) = \sum_{l \geq 0} \frac{(x + \log\frac{y_2}{m})^l}{l!} \frac{1}{2 \pi i} \oint \frac{s^{i-l-2}}{1+\alpha s} ds.
\end{equation*}
Taking a power series of $(1+\alpha s)^{-1}$, we get
\begin{equation*}
N_1(x) = \sum_{l \geq 0} \frac{(x + \log\frac{y_2}{m})^l}{l!} \sum_{k \geq 0} (-\alpha)^k \frac{1}{2 \pi i} \oint s^{k+i-l-2} ds.
\end{equation*}
The integral picks out $l = k+i-1$, giving
\begin{equation*}
N_1(x) = (x + \log\frac{y_2}{m})^{i-1} \sum_{k \geq 0} \frac{(-\alpha)^k (x + \log\frac{y_2}{m})^{k}}{(k+i-1)!}.
\end{equation*}
Now separate the variables $i$ and $k$ by the following trick (with $B(x,y)$ denoting the standard beta function)
\begin{equation*}
\frac{1}{(k+i-1)!} = B(i-1,k+1) \frac{1}{k! (i-2)!}
\end{equation*}
followed by the usual integral representation definition of the beta function, getting
\begin{equation*}
N_1(x) = \frac{(x + \log\frac{y_2}{m})^{i-1}}{(i-2)!} \int_0^{1} (1-u)^{i-2} \sum_{k \geq 0} \frac{(-\alpha)^k u^k (x + \log\frac{y_2}{m})^{k}}{k!}  du.
\end{equation*}
Of course the sum over $k$ is easily computable which completes the proof.
\end{proof}
Applying Lemmas \ref{lemma:L1computed} and \ref{lemma:L1exact} to \eqref{eq:J2'} gives
\begin{multline}
\label{eq:J2'divisor}
J_2' = \frac{16}{(i-2)!(j-2)!} \sum_{m \leq y_2} \frac{d_5(m)}{m} \frac{d^4}{dx^2 dy^2} \left[e^{\beta x+ \alpha y} (x+\log(y_2/m))^{i-1}(y+\log(y_2/m) )^{j-1} 
\right.
\\
\int_0^{1} \int_0^{1} (1-u)^{i-2} (1-v)^{j-2} e^{-\alpha x u - \beta y v} \leg{y_2}{m}^{-\alpha u - \beta v} du dv
 \Big]_{x=y=0}
\\
+ O(L^{i+j-7} \sum_{m \leq y_2} \frac{d_5(m)}{m}) 
+ O((L^{i-3} + L^{j-3}) L^{\varepsilon} \sum_{m \leq y_2} \frac{d_5(m)}{m} \leg{y_2}{m}^{-\nu}).
\end{multline}
Applying Lemma \ref{lemma:logsave} and noting $\max(i+1, j+1) \leq i+j-2$ since $i, j \geq 3$, shows that the error terms above are
\begin{equation*}
\ll L^{i+j-2} + L^{\varepsilon}(L^{i+1} + L^{j+1}) \ll L^{i+j-2+\varepsilon}.
\end{equation*}
Thus inserting \eqref{eq:J2'divisor} into \eqref{eq:I2'} and recalling $I_{2}'(\alpha, \beta) = I_{20}'(\alpha, \beta) + O(T^{1-\varepsilon})$ gives
\begin{multline*}
 I_2'(\alpha,\beta) = \frac{16 \widehat{w}(0)}{\alpha+\beta} \frac{d^4}{dx^2 dy^2} \Big[ \int_0^{1} \int_0^{1} e^{x(\beta -\alpha u) + y(\alpha -\beta v)} 
\sum_{m \leq y_2} \frac{d_5(m)}{m} 
\frac{(x + \log \frac{y_2}{m})(y + \log \frac{y_2}{m})}{(\log y_2)^4}
\\ 
P_2''\Big((1-u)\frac{x + \log \frac{y_2}{m}}{\log y_2}\Big) P_2''\Big((1-v)\frac{y + \log \frac{y_2}{m}}{\log y_2}\Big) \leg{y_2}{m}^{-\alpha u - \beta v} du dv
\Big]_{x=y=0} + O(TL^{-1+\varepsilon}).
\end{multline*}
We write this main term in a more convenient way as
\begin{multline*}
\frac{16 \widehat{w}(0)}{(\alpha+\beta)(\log y_2)^6} \frac{d^4}{dx^2 dy^2} \Big[ \int_0^{1} \int_0^{1} y_2^{x(\beta -\alpha u) + y(\alpha -\beta v)} 
\sum_{m \leq y_2} \frac{d_5(m)}{m} \leg{y_2}{m}^{-\alpha u - \beta v}
\\ 
\Big(x+\frac{\log \frac{y_2}{m}}{\log y_2}\Big)\Big(y+\frac{\log \frac{y_2}{m}}{\log y_2}\Big)
P_2''\Big((1-u)(x+\frac{\log \frac{y_2}{m}}{\log y_2})\Big) P_2''\Big((1-v)(y+\frac{\log \frac{y_2}{m}}{\log y_2})\Big)  du dv
\Big]_{x=y=0}. 
\end{multline*}
Now apply Corollary \ref{coro:EulerMacdiag} to this sum over $m$ to get
\begin{multline*}
 I_2'(\alpha,\beta) = \frac{2 \widehat{w}(0)}{3(\alpha+\beta)(\log y_2)}  \frac{d^4}{dx^2 dy^2}
\Big[
\int_0^{1} \int_0^{1} \int_0^{1} (1-r)^{4} y_2^{\beta(x-v(y+r)) + \alpha(y-u(x+r))} 
\\
 (x+r)(y+r) P_2''\left((1-u)(x + r)\right) P_2''\left((1-v)(y + r)\right) dr du dv
\Big]_{x=y=0} + O(TL^{-1+\varepsilon}).
\end{multline*}
To form $I_2(\alpha,\beta)$, recall that we need to add $I_2'$ and $I_2''$, where $I_2''$ is formed by taking $I_2'$, switching $\alpha$ and $-\beta$, and multiplying by $T^{-\alpha-\beta}$.  Letting
\begin{equation}
 U(\alpha,\beta) = \frac{y_2^{\beta(x-v(y+r)) + \alpha(y-u(x+r))} - T^{-\alpha-\beta} y_2^{-\alpha(x-v(y+r)) - \beta(y-u(x+r))}}{\alpha + \beta},
\end{equation}
we then have
\begin{multline}
 I_2(\alpha,\beta) = \frac{2 \widehat{w}(0)}{3(\log y_2)}  \frac{d^4}{dx^2 dy^2}
\left[ \int_0^{1} \int_0^{1} \int_0^{1} (1-r)^{4} U(\alpha,\beta)
\right.
\\
\left.  (x+r)(y+r) P_2''\left((1-u)(x + r)\right) P_2''\left((1-v)(y + r)\right) dr du dv
\right]_{x=y=0} + O(TL^{-1+\varepsilon}).
\end{multline}
Now write
\begin{equation*}
U(\alpha,\beta) = y_2^{\beta(x-v(y+r)) + \alpha(y-u(x+r))} \frac{1  - (Ty_2^{x+y - v(y+r) - u(x+r)})^{-\alpha-\beta}}{\alpha + \beta}.
\end{equation*}
Using the integral formula
\begin{equation*}
 \frac{1-z^{-\alpha-\beta}}{\alpha + \beta} = \log{z} \int_0^{1} z^{-t(\alpha + \beta)} dt, 
\end{equation*}
and simplifying, we finish the proof of Lemma \ref{lemma:mainterm2}.

\subsection{The arithmetical factor}
Here we compute $B(0,0,0)$.  The first thing to notice is that since $B$ is holomorphic with respect to $\alpha, \beta$ near the origin, we have $B(0,0,0) = B_0(0,0,0) (1 + O(L^{-1}))$, where $B_0$ is $B$ specialized with $\alpha=\beta = 0$.  For this choice of $\alpha,\beta$ and with $u=z=s$, the expression \eqref{eq:B} simplifies greatly.  Thus
\begin{equation*}
 B_{0}(s,s,s) = \sum_{h_1 k_2 m = h_2 k_1 n} \frac{\mu_2(h_1) \mu_2(h_2) }{(h_1 k_1 h_2 k_2 mn)^{\half+s}}.
\end{equation*}
Making $h_1 k_2 = l_1$ and $h_2 k_1 = l_2$ be new variables and using $\sum_{h | l} \mu_2(h) = \mu(l)$, we get
\begin{equation*}
 B_{0}(s,s,s) = \sum_{l_1 m = l_2 n} \frac{\mu(l_1) \mu(l_2) }{(l_1 l_2 mn)^{\half+s}} = 1.
\end{equation*}

\end{document}